\documentclass[11pt]{article}
\usepackage{graphicx}
\usepackage{hhline}
\usepackage{ifxetex} 
\usepackage{setspace} 
\usepackage{amsmath,amsthm,amssymb,amsfonts}
\usepackage{amsmath}
         \usepackage{makeidx}
         \usepackage{ifxetex} 
\usepackage[all]{xy}
\usepackage[pagebackref=true,colorlinks,linkcolor=blue,citecolor=magenta]{hyperref}
\usepackage{fancybox}
\usepackage{fancyhdr}
\usepackage{tocbibind}
\usepackage{makeidx}
\makeindex
\makeatletter \oddsidemargin .4in \evensidemargin -.1in \textwidth
14.5cm \topmargin 0.1cm \textheight 22cm
\newcommand{\doublespacinggg}{\let\CS=\@currsize\renewcommand{\baselinestretch}{1.55}\tiny\CS}
\newcommand{\doublespacingg}{\let\CS=\@currsize\renewcommand{\baselinestretch}{1.75}\tiny\CS}
\doublespacinggg

\newtheorem{thm}{Theorem}[section]
 \newtheorem{cor}[thm]{Corollary}
 
   \newtheorem{examp}[thm]{Example}
  \newtheorem{defin}[thm]{Definition}
  \newtheorem{prop}[thm]{Proposition}
  \newtheorem{rem}[thm]{Remark}

\newcommand{\be}{\begin{equation}}
\newcommand{\ee}{\end{equation}}
\newcommand{\bea}{\begin{eqnarray}}
\newcommand{\eea}{\end{eqnarray}}

\newcommand{\bee}{\begin{eqnarray*}}
\newcommand{\eee}{\end{eqnarray*}}

\newcommand{\0}{${\bf 0}$}

\title{On a generalization of McCoy Rings}
\author{ \small Mohammad Vahdani Mehrabadi $^{*}$, 
Shervin Sahebi $^{**}$ and Hamid H. S. Javadi $^{***}$\\ $^{*,**}$ 
Department of Mathematics, Islamic Azad University,\\
Central Tehran Branch, 13185/768, Iran, \\email: 
md \underline{~}vahdani@yahoo.com; sahebi@iauctb.ac.ir 
\\ 
$^{***}$ Department of Mathematics and Computer Science, Shahed University,\\ Tehran, Iran, email: h.s.javadi@shahed.ac.ir.\\}

\begin{document}

\date{}
\maketitle \noindent \vspace{-.8cm}

\doublespacingg

\begin{center}
\begin{minipage}{11cm} \footnotesize { \textsc{Abstract:}
We introduce Central McCoy rings, which are a generalization of McCoy rings and investigate their properties. For a ring $R$, we prove that $R$ is right Central McCoy if and only if the polynomial ring $R[x]$ is right Central McCoy. Also, we give some examples to show that if $R$ is right Central McCoy, then $M_{n}(R)$ and $T_{n}(R)$ are not necessary right Central McCoy, but $D_{n}(R)$ and $V_{n}(R)$ are right Central McCoy, where $D_{n}(R)$ and $V_{n}(R)$ are the subrings of the triangular matrices with constant main diagonal and constant main diagonals, respectively.
  }

 \end{minipage}
\end{center}

 \vspace*{.4cm}

 \noindent {\footnotesize {\bf Mathematics Subject Classification 2010:} 16U20, 16S36, 16W20.    }

\vspace*{.2cm}

\doublespacing

\section{Introduction}
Throughout this paper, $R$ denotes an associative ring with identity. 
For notation $M_{n}(R)$, $T_{n}(R)$, $S_{n}(R)$, $R[x]$ and $C(R)$ denote, 
the $n \times n$ matrix ring over $R$, upper triangular matrix ring over $R$, 
diagonal matrix over $R$, polynomial ring over $R$ and center of a ring $R$, respectively. 
We denote the identity matrix and unit matrices in ring $M_n(R)$, 
by $I_n$ and $E_{ij}$, respectively. 
Rege -Chhawchharia [13] called a noncommutative ring $R$ right McCoy 
if whenever polynomials 
$f(x)=\sum_{i=0}^n a_{i}x^{i}$, 
$g(x)=\sum_{j=0}^m b_{j}x^{j}\in R[x]\setminus\{0\}$  
satisfy $f(x)g(x)=0$, there exists nonzero elements $r\in R$ 
~such ~that~ $a_{i}r=0$. 
Left McCoy rings are defined similarly. 
A number of papers have been written on the McCoy property of ring 
(see, e. g., [2, 4, 8, 10, 12, 14]). In [12], 
it is shown that there exists a left McCoy ring but not right McCoy. 
The name "McCoy" was chosen because 
McCoy [11] had noted that every commutative ring satisfies this condition. 
In [10], it is shown that $R$ is McCoy rings if and only if $R[x]$ is McCoy. 
In [2], it is shown that if $e$ is a central idempotent element of $R$, 
then $R$ is a right McCoy ring if and only if 
$eR$ is a right McCoy ring if and only if 
$(1-e)R$ is a right McCoy ring. 
Also in [14], it is shown, 
if there exists the classical right quotient ring $Q$ of a ring $R$, 
then $R$ is right McCoy if and only if $Q$ is right McCoy. 
Armendariz rings give another family of McCoy rings. 
A ring $R$ is called Armendariz [13] if whenever 
polynomials $f(x)=\sum_{i=0}^n a_{i}x^{i}$, $g(x)=\sum_{j=0}^m b_{j}x^{j}\in R[x]\setminus\{0\}$  satisfy $f(x)g(x)=0$, 
then $a_{i}b_{j}=0$ for each $i$ and $j.$ 
 Agayev, G$\ddot{u}$ng$\ddot{o}$glu, 
 Harmonic and Halicio$\breve{g}$lu [1] called a ring $R$ Central Armendariz 
 if whenever polynomials $f(x)=\sum_{i=0}^n a_{i}x^{i}$, 
 $g(x)=\sum_{j=0}^m b_{j}x^{j}\in R[x]\setminus\{0\}$  
 satisfy $f(x)g(x)=0$, then $a_{i}b_{j}\in C(R) $ for all $i$ and $j.$ 
 They showed that the class of Central Armendariz rings 
 lies strictly between classes of Armendariz rings and abelian rings 
 (That is every idempotent of it belong to $C(R)$).

Motivated by the above results, 
we investigate a generalization of right McCoy rings, 
which we call it right Central McCoy. 
We say that a ring $R$ is right Central McCoy 
(respectively left Central McCoy) if for each pair of nonzero polynomials 
$f(x)=\sum_{i=0}^n a_{i}x^{i}$, $g(x)=\sum_{j=0}^m b_{j}x^{j}\in R[x]$, 
with  $f(x)g(x)=0$, then there exists a nonzero element 
$r \in R$ with $a_{i}r \in C(R)$ (respectively $rb_{j} \in C(R)$). 
A ring is central McCoy if it is both left and right Central McCoy. 
It is clear that McCoy rings are Central McCoy, 
but the converse is not always true. 
According to Cohn [6], a ring $R$ is called reversible 
if $ab=0$ implies $ba=0$, for $a , b \in R$. 
A ring $R$ is said to be semicommutative if for all $a , b \in R$  
if $ab=0$ implies $aRb=0$. 
Clearly, reduced rings are reversible and reversible rings are also semicommutative. 
In [12, Theorem 2], Nielsen showed that 
reversible rings are McCoy and therefore are Central McCoy. 
Hence reduced rings are Central McCoy.
  
  \section{Central McCoy Rings}

 \noindent
We start this section by the following definition:
\begin{defin}\label{def 1.1}  
A ring $R$ is said to be right Central McCoy 
(respectively left Central McCoy) 
if for each pair of nonzero polynomials $f(x)=\sum_{i=0}^n a_{i}x^{i}$, 
$g(x)=\sum_{j=0}^m b_{j}x^{j}\in R[x]\setminus\{0\}$, 
with  $f(x)g(x)=0$, then there exists a nonzero element $r \in R$ 
with $a_{i}r \in C(R)$ (respectively $rb_{j} \in C(R)$). 
A ring $R$ is called Central McCoy if it is both left and right Central McCoy. 
\end{defin}

It is clear that right (resp. left) McCoy rings are right (resp. left) Central McCoy, 
but the converse is not always true by the following example.

\begin{examp}\label{examp1} 
Let $K$ be a field and $K<x,y,z>$ be the free algebra 
with noncommuting indeterminates $x,y,z$ over $K$. 
Set $R$ be the factor ring of $K<x,y,z>$ with relations
 \begin{center}
$x^{2}=yx=x~~ , ~~z^{2}=0~~ , ~~y^{2}=xy=y~~ , ~~zx=xz=yz=zy=z$
\end{center} 
We coincide $\{x,y,z\}$ with their images in $R$, for simplicity. 
Consider the subring of $R$ generated by $\{\alpha, x, y, z| \alpha \in K \}$, say $S$. 
Then every element of $S$ is of the form, 
$\alpha_{1}+\alpha_{2}x+\alpha_{3}y+\alpha_{4}z$ 
with $\alpha_{i}$'s in the field $K$. 
By the construction of $S$, 
we have $(x+yt)((1-x)+(1-y)t)=0$ while $x+yt$ 
and $(1-x)+(1-y)t$ are nonzero polynomials over $S$. 
Assume $(x+yt)(\alpha_{1}+\alpha_{2}x+\alpha_{3}y+\alpha_{4}z)=0$. 
Then $x(\alpha_{1}+\alpha_{2}x+\alpha_{3}y+\alpha_{4}z)=\alpha_{1}x+\alpha_{2}x+\alpha_{3}y+\alpha_{4}z=0$ implies $\alpha_{3} =\alpha_{4}=0$ 
and $y(\alpha_{1}+\alpha_{2}x)=\alpha_{1}y+\alpha_{2}x=0$ 
implies $\alpha_{1}=\alpha_{2}=0$. 
Thus $S$ can not be right McCoy.\\
Next we show that $S$ is right Central McCoy. 
Let $f(t)=\sum_{i=0}^n (\alpha_{1i}+\alpha_{2i}x+\alpha_{3i}y+\alpha_{4i}z)t^{i}$ and $g(t)=\sum_{j=0}^m (\beta_{1j}+\beta_{2j}x+\beta_{3j}y+\beta_{4j}z)t^{j}$ 
be nonzero in $S[t]$ with $f(t)g(t)=0$. Then $(\alpha_{1i}+\alpha_{2i}x+\alpha_{3i}y+\alpha_{4i}z)z=
(\alpha_{1i}+\alpha_{2i}+\alpha_{3i})z \in C(S)$, 
since $(\alpha_{1i}+\alpha_{2i}+\alpha_{3i})z (\alpha_{1}+\alpha_{2}x+\alpha_{3}y+\alpha_{4}z)=
(\alpha_{1i}+\alpha_{2i}+\alpha_{3i})(\alpha_{1}+\alpha_{2}+\alpha_{3})z=
 (\alpha_{1}+\alpha_{2}x+\alpha_{3}y+\alpha_{4}z)(\alpha_{1i}+\alpha_{2i}+\alpha_{3i})z$ 
for each $(\alpha_{1}+\alpha_{2}x+\alpha_{3}y+\alpha_{4}z) \in S$.
\end{examp}

The following example shows that, Central McCoy condition is not left-right symmetric.
\begin{examp}\label{examp} Let $K$ be a field and 
$A=K<a_0, b_0, a_1, b_1>$ be the free algebra 
by the noncommuting indeterminates 
$a_{0},b_{0},a_{1},b_{1}.$ 
Let $I$ be the ideal of $A$ generated by \begin{center}
$a_{0}b_{0}, a_{0}b_{1}+a_{1}b_{0}, a_{1}b_{1}, b_{s}b_{t}$
\end{center}
 for $s,t \in \{0,1\}$ and let $R= A/I$. 
 We identify $a_{i}$ and $b_{j}$ with their images in $R$ for simplicity.
 By the construction of $R$, 
 we have $(a_{0}+a_{1}x)(b_{0}+b_{1}x)=0$ 
 while $a_{0}+a_{1}x$ and $a_{0}+a_{1}x$ are nonzero polynomials over $R$. 
 Assume by way of contradiction that 
 there exists $0\neq r \in R$ such that $a_{0}r, a_{1}r \in C(R)$. 
 Therefore, $a_0rb_0=b_0a_0r.$ 
 A computation using the reduced forms for elements in $R$ 
 shows that $a_{0}r=0=a_{1}r$, 
 which quickly implies $r=0$, a contradiction. 
 This yields that $R$ is not right Central McCoy. \\
Next we show that $R$ is a  left Central McCoy. 
Let $f(x)$ and $g(x)$ be nonzero in $R[x]$ with 
$f(x)g(x)=0$. 
Then $b_{j}g(x)=0\in C(R)[x]$ for $j=0,1$, by [5, Example 2]. 
Thus $R$ is  left Central McCoy.
\end{examp}

Now, we get some of the basic properties of central McCoy rings. 
\begin{prop}\label{proposition1}
Let $R_{k}$ be a ring, where $k \in I$. Then $R_{k}$ is right (resp. left) Central McCoy for each $k \in I$ if and only if $R=\prod_{k\in I} R_{k}$ is right (resp. left) Central McCoy.
\end{prop}
\begin{proof}
Let each $R_{k}$ be a right Central McCoy ring and 
$f(x)= \sum_{i=0}^m a_{i}x^{i}, 
g(x)= \sum_{j=0}^n b_{j}x^{j} \in R[x]\setminus\{0\}$ 
such that $f(x)g(x)=0$, where $a_{i}=(a_{i}^{~(k)})$, 
$b_{j}=(b_{j}^{~(k)})$. 
If there exists $t\in I$ such that $a^{(t)}_{i}=0$ 
for each $0\leq i \leq m$, 
then we have $a_{i}c=0 \in C(R)$ 
where $c=(0,0,\ldots, 1_{R_{t}},0,\ldots,0)$. 
Now suppose for each $k\in I$, 
there exists $0\leq i_{k} \leq m$ 
such that $a^{(k)}_{i_{k}}\neq 0$. 
Since $g(x)\neq 0$, there exists $t\in I$ 
and $0\leq j_{t} \leq n$ such that $ b^{(t)}_{j_{t}}\neq 0$. 
Consider $f_{t}(x)=\sum_{i=0}^m a^{(t)}_{i}x^{i}$ 
and $g_{t}(x)=\sum_{i=0}^n b^{(t)}_{j}x^{j} \in R_{t}[x]\setminus\{0\}$. 
We have $f_{t}(x)g_{t}(x)=0$. 
Thus there exists nonzero $c_{t}\in R_{t}$ 
such that $a^{(t)}_{i}c_{t}\in C(R_{t})$, 
for each $0\leq i \leq m$, since $R_{t}$ is right Central McCoy ring. 
Therefore $a_{i}c\in C(R)$, for each $0\leq i \leq m$, 
where $c=(0,\ldots,0,c_{t},0,\ldots,0)$ and so $R$ is right Central McCoy.\\
Conversely, suppose $R$ is right Central McCoy and $t\in I$. 
Let $f(x)= \sum_{i=0}^m a_{i}x^{i}, 
g(x)= \sum_{j=0}^n b_{j}x^{j}$ be nonzero polynomials in $R_{t}[x]$ 
such that $f(x)g(x)=0$. 
Let $$F(x)=\sum_{i=0}^m (0,0,\ldots,0,a_{i},0,\ldots,0)x^{i},~G(x)=\sum_{j=0}^n (0,0,\ldots,0,b_{j},0,\ldots,0)x^{j} \in R[x]\setminus\{0\}.$$ 
Hence $F(x)G(x)=0$ and so there exists 
$0\neq c=(c_{1},c_{2},\ldots,,c_{t-1},c_{t},c_{t+1},\ldots,c_{m-1},c_{m})$ 
such that $(0,0,\ldots,0,a_{i},0,\ldots,0)c\in C(R)$. 
Therefore,  $a_{i}c_{t}\in C(R_{t})$ and so $R_{t}$ is right Central McCoy.
\end{proof}

\begin{cor}\label{Corollary}
Let $D$ be a ring and $C$ a subring of $D$ with $1_{D}\in C$. 
Let  
\begin{center}
$R(C,D) = \{(d_{1}, \ldots,d_{n}, c, c, \ldots) \mid d_{i} \in D, c \in C, n\geq1\}$ 
\end{center} with addition and multiplication defined component-wise, 
$R(D,C)$ is a ring. 
Then $D$ is right (resp. left) Central McCoy if and only if 
$R(D,C)$ is right (resp. left) Central McCoy.
\end{cor}

\begin{thm}\label{theorem0}
For a ring $R$, $R[x]$ is right (resp. left) Central McCoy 
if and only if $R$ is right (resp. left) Central McCoy.
\end{thm}\label{theorem1}
\begin{proof}
Suppose that $R$ is right Central McCoy. 
Let $R[x][t]$ denote the polynomial ring 
with an indeterminate t over $R[x]$. 
Let $f(t)g(t)=0$ for nonzero polynomial 
$f(t)=\sum_{i=0}^m f_{i}(x)t^{i}~,~ g(t)=\sum_{j=0}^n g_{j}(x)t^{j} \in R[x][t].$ Let $k=\sum_{i=0}^m deg f_{i}(x) + \sum_{j=0}^n degg_{j}(x).$ 
The degree of the zero polynomial is taken to be zero. 
Let $F(x)=f_{0}+f_{1}x^{k}+\ldots+f_{m}x^{k{m}}$,  $G(x)=g_{0}+g_{1}x^{k}+\ldots+g_{n}x^{k{n}}\in R[x]\setminus\{0\}.$ 
So $F(x)G(x)=0$ since $f(t)g(t)=0$. 
Then there exists $c\neq0 \in R$ such  that  
$a_{i}c\in C(R)$ for all $0\leq i \leq m$, 
and so $f_{i}(x)c\in C(R[x]).$\\
Conversely,  suppose  that $R[x]$  is  right Central  McCoy. 
Let $f(x)g(x)=0$ ~for ~nonzero \\
polynomial~$f(x) =\sum_{i=1}^n a_{i}x^{i}$ 
and  $g(x)=\sum_{j=1}^m b_{j}x^{j} \in R[x]\setminus\{0\}.$ 
Suppose $f(t)$ and $g(t)$ is constant polynomial in $R[x][t]$ 
such that $f(t)g(t)=0$. 
Then there exists $0\neq c(x) = c_{0}+c_{1}x+...+c_{k}x^{k}\in R[x]$ 
such that $f(t)c(x) \in C(R[x])$. 
So $a_{i}c(x)\in C(R[x])$. 
Since $c(x)$ is a nonzero element, 
then at least one of the $c_{i}$' is nonzero element, 
for example $c_{t}$. So $a_{i}c_{t}\in C(R)$. 
Therefore $R$ is right Central McCoy.
 \end{proof}
 
\begin{prop}\label{proposition}
Let $R$ be a ring and $e$ a central idempotent element of $R$. 
Then $R$ is a right (resp. left) Central McCoy ring if and only if 
$eR$ is a right (resp. left) Central McCoy ring if and only if 
$(1-e)R$ is right (resp. left) Central McCoy.
\end{prop}
\begin{proof}
Assume that $R$ is a right Central McCoy ring and consider 
$f(x)= \sum_{i=0}^n ea_{i}x^{i}$,
 $g(x)= \sum_{j=0}^m eb_{j}x^{j} \in eR[x]\setminus\{0\} \subseteq R$ 
 such that $f(x)g(x)=0$. 
 Since $R$ is a right Central McCoy ring, 
 there exists $s \in R$ such that $(ea_{i}) s  \in C(R)$. 
 So $(ea_{i}) sr=r(ea_{i}) s$, for any $r \in R$. 
 Therefore $((ea_{i})(es))(er)=er(ea_{i})(es))$. 
 So $(ea_{i})(es) \in C(eR)$. 
 Hence $eR$ is right Central McCoy. 
 Similarly, we prove that $(1-e)R$ is a right Central McCoy ring.\\
 Conversely, assume that ~$eR$ is a right Central McCoy ring. 
 Consider ~$f(x)= \sum_{i=0}^n a_{i}x^{i}$, 
 $g(x)= \sum_{j=0}^m b_{j}x^{j} \in R[x]\setminus\{0\}$ such that $f(x)g(x)=0$. 
 Clearly $ef(x), eg(x) \in eR[x]$ and 
 $(ef(x))(eg(x))=e^{2}f(x)g(x)=ef(x)g(x)=0$, 
 since $e$ is a central idempotent element of $R$. 
 Then there exists $s \in eR$ such that $(ea_{i})s \in C(eR)$. 
 So $(ea_{i})ser = er(ea_{i})s$, for any $r \in R$. 
 Therefore $(a_{i}se)r=r(a_{i}es)$. 
 So $a_{i}s \in C(R)$. 
 Hence, $R$ is right Central McCoy. 
 Similarly, this fact is satisfied if $(1-e)R$ is a right Central McCoy ring.
\end{proof}
The following example shows that, 
if $R$ is a right Central McCoy ring, 
then $M_{n}(R)$ and $T_{n}(R)$ are  not necessary right Central McCoy.

\begin{examp}\label{examp} Let ~$R$~be~ a~ commutative ring, 
then R is right Central McCoy ring. 
Let $f(x)=E_{12}+E_{11}x,  g(x)=E_{12}-E_{22}x \in (M_{2}(R))[x]$. 
Then~ $f(x)g(x)=0$. 
If there exists $P=\sum_{i=1}^2 \sum_{j=1}^2P_{ij}E_{ij} \in M_{2}(R)$ 
such that $f(x)P \in C(M_{2}(R))$, 
then $E_{12}P=P_{21}E_{11}+P_{22}E_{12} \in C(M_{2}(R))$. 
Therefore, $(P_{21}E_{11}+P_{22}E_{12})E_{12}=E_{12}(P_{21}E_{11}+P_{22}E_{12})$ 
and so $P_{21}= 0.$ 
Also $(P_{21}E_{11}+P_{22}E_{12})E_{22}=E_{22}(P_{21}E_{11}+P_{22}E_{12})$~ 
and so $P_{22}= 0.$ 
Similarly, $P_{11}= P_{12} = 0$. 
Therefore $P=0$. By choosing $f(x)$ and 
$g(x)$ the same as above and by similar argument, 
we can show that $T_{2}(R)$ is not right Central McCoy.
\end{examp}
 
\begin{thm}\label{theorem0}
For a ring $R$, we have $R$ is right (resp. left) Central McCoy 
if and only if one of the following holds:\\
(1) $ D_{n}(R)= \bigg\{ \left(\begin{array}{ccccc}
a & a_{12} & a_{13} & \ldots & a_{1n} \\  0 & a & a_{23} & \ldots & a_{2n} \\  0 & 0 & a & \ldots & a_{3n} \\  \vdots & \vdots & \vdots & \ddots & \vdots \\  0 & 0 & 0 & \ldots & a
\end{array}\right)| a, a_{ij} \in R\bigg\}$ is right (resp. left) Central McCoy for any $n\geq1.$\\
(2) $V_{n}(R)= \bigg\{ \left(\begin{array}{cccccc}
 a_{1} & a_{2} & a_{3} & a_{4} & \ldots & a_{n} \\
    0 & a_{1} & a_{2} & a_{3}& \ldots & a_{n-1} \\
    0 & 0 & a_{1} & a_{2} & \ldots & a_{n-2} \\
    \vdots & \vdots & \vdots & \vdots & \ldots & \vdots \\
    0 & 0 & 0 & 0 & \ldots & a_{2} \\
    0 & 0 & 0 & 0 & \ldots & a_{1} \\
\end{array}\right) | a_{1}, a_{2}, . . . , a_{n}\in R\bigg\}\cong \frac{R[x]}{(x^{n})}$ 
is right (resp. left) Central McCoy for any $n\geq1,$ 
where $(x^{n})$ is a two- sided ideal of $R[x]$ generated by $x^{n}.$
\end{thm}\label{theorem1}

\begin{proof}
(1) Let $F(x)=\sum_{i=0}^p A_{i}x^{i}, G(x)=\sum_{j=0}^q B_{j}x^{j} \in D_{n}(R)[x] $ 
where,
$$A_{i}=\left(\begin{array}{cccc}
a_{i} & a^{i}_{12}& \ldots & a^{i}_{1n} \\  0 & a_{i} &\ldots & a^{i}_{2n} \\  \vdots & \vdots & \ddots & \vdots \\ 0 & 0& \ldots & a_{i}
\end{array}\right), B_{j}=\left(\begin{array}{cccc}
b_{j} & b^{j}_{12}& \ldots & b^{j}_{1n} \\   0 & b_{j} &\ldots & b^{j}_{2n} \\  \vdots & \vdots & \ddots & \vdots \\ 0 & 0& \ldots & b_{j}
\end{array}\right)$$
Then
\begin{center}
$F(x)=\left(\begin{array}{cccc}
f(x) & f_{12}(x)& \ldots &f_{1n}(x) \\   0 & f(x) &\ldots & f_{2n}(x) \\  \vdots & \vdots & \ddots & \vdots \\0 & 0& \ldots & f(x)
\end{array}\right), G(x)=\left(\begin{array}{cccc}
 g(x) & g_{12}(x)& \ldots &g_{1n}(x) \\ 0 & g(x) &\ldots & g_{2n}(x) \\   \vdots & \vdots & \ddots & \vdots \\ 0 & 0& \ldots & g(x)
\end{array}\right)$
\end{center}
where $f(x)=\sum_{i=0}^pa_{i}x^{i}$, $f_{kl}(x)=\sum_{i=0}^pa^{i}_{kl}x^{i}$, $g(x)=\sum_{j=0}^qb_{j}x^{j}$ 
and $g_{kl}(x)=\sum_{j=0}^qb^{j}_{kl}x^{j}$ 
for any $k=1,2,...,n$, $l=2,3,...,n$ and~$ k<l$. 
Suppose $F(x)G(x)=0$, and $F(x),G(x)\neq0$. 
Set $H(x)=F(x)G(x)=(h_{pq}(x))$ for~~$ p,q=1,2,...,n$.
\\
CASE 1. If $f(x)\neq0$, $g(x)\neq0$, 
then $h_{11}(x)=f(x)g(x)=0$. 
Since $R$ is right Central McCoy there exists $r\in R\setminus\{0\}$ 
such that $a_{i}r\in C(R)$. Let $A=rE_{1n}$. 
Then $A_{i}C\in C(D_{n}(R))$.
\\
CASE 2.~If $f(x)\neq0$ and $g(x)=0$, then there exists $g_{kl}(x)\neq0$, such that $g_{(k+u)l}(x)=0$ for some $k,l$, and $1 \leq u \leq n-k$, since $G(x)\neq0$. So~$ h_{kl}(x)=f(x)g_{kl}(x)=0$. Hence there exists $r\in R\setminus\{0\}$ such that $a_{i}r\in C(R)$. Let $A=rE_{1n}$. Then $A_{i}A\in C(D_{n}(R))$.
\\
CASE 3. If $f(x)=0$, then clearly $A_{i}A=0$, where $A=E_{1n}$. Thus $D_{n}(R)$ is right Central McCoy.
\\
Conversely, assume that $f(x)g(x)=0$ , where 
$f(x)=\sum_{i=0}^na_{i}x^{i}$, $g(x)=\sum_{j=0}^mb_{j}x^{j}$ are nonzero polynomial of R[x]. Let $F(x)=\sum_{i=0}^nA_{i}x^{i}$  ,  $G(x)=\sum_{j=0}^mB_{j}x^{j}$, where 
$$A_{i}=\left(\begin{array}{cccc}
 a_{i} & a_{i}& \ldots & a_{i} \\ 0 & a_{i} &\ldots & a_{i} \\ \vdots & \vdots & \ddots & \vdots \\ 0 & 0& \ldots & a_{i}
\end{array}\right), B_{j}=\left(\begin{array}{cccc}
 b_{j} & b_{j}& \ldots & b_{j} \\ 0 & b_{j} &\ldots & b_{j} \\ \vdots & \vdots & \ddots & \vdots \\ 0 & 0& \ldots & b_{j}
\end{array}\right)$$ for any $i=0,1,...,n,j=0,1,...,m$. Then 
$$F(x)G(x)=\left(\begin{array}{cccc}
 f(x) & f(x)& \ldots & f(x) \\ 0 & f(x) &\ldots &f(x) \\ \vdots & \vdots & \ddots & \vdots \\ 0 & 0& \ldots & f(x)
\end{array}\right)\left(\begin{array}{cccc}
g(x) & g(x)& \ldots & g(x) \\0 & g(x) &\ldots & g(x) \\ \vdots & \vdots & \ddots & \vdots \\  0 & 0& \ldots & g(x)
\end{array}\right)=0.$$
Hence there exists $A$=$\left(\begin{array}{cccc}
s & s_{12}& \ldots & s_{1n} \\ 0 & s &\ldots &s_{2n} \\ \vdots & \vdots & \ddots & \vdots \\  0 & 0& \ldots & s
\end{array}\right) \in D_{n}(R)\setminus\{0\}$ such that $A_{i}A\in C(D_{n}(R))$, since $D_{n}(R)$ is right Central McCoy. If $s\neq0$, then $a_{i}s\in C(R)$. If $s=0$, then there exists $s_{ij}\neq 0$ for some $i,j$, such that $s_{i+v,j}=0$ for any $ 1 \leq v \leq n-i$. We also have $a_{i}s_{ij}\in C(R)$. Thus, $R$ is right Central McCoy.\\
(2) The proof is similar to (1). 
\end{proof}

\begin{rem}\label{rem1} It is natural to ask whether $R$ is a right (resp. left) Central McCoy ring if for any nonzero proper ideal $I$ of $R$, $R/I$ and I are right (resp. left) Central McCoy, where $I$ is considered as a right (resp. left) Central McCoy ring without identity. However, we have a negative answer to this question by the following example:
\end{rem}

\begin{examp}\label{examp1} Let  F be a field and consider $R$=$T_{2}(F)$, which is not right Central McCoy by Example 2.8. But by [8, Example 3.2] $R/I$ and $I$ are McCoy and therefore are Central McCoy. 
\end{examp}

Let $S$ denote a multiplicatively closed subset of a ring $R$ consisting of central regular elements. Let $RS^{-1}$ be the localization of $R$ at $S$. Then we have:

\begin{thm}\label{theorem0} A ring $R$ is right (resp. left) Central McCoy if and only if $RS^{-1}$ is right (resp. left) Central McCoy.
\end{thm}\label{theorem1}

\begin{proof}
 Suppose that $R$ is right Central McCoy. Let $f(x)=\sum_{i=0}^n a_{i}s_{i}^{-1}x^{i}$ and  $g(x)=\sum_{j=0}^m b_{j}t_{j}^{-1}x^{j} \in (RS^{-1})[x]\setminus\{0\}$  such that $f(x)g(x)=0$. Then we may find $u, v, c_{i}$ and $d_{j} \in S$ such that $uf(x)= \sum_{i=0}^n a_{i}c_{i}x^{i}$ , $vg(x)=\sum_{j=0}^m b_{j} d_{j} x^{j} \in R[x]$ and $(uf(x))(vg(x))=0$. By supposition, there exists $r \in R$ such that  $a_{i}c_{i}r \in C(R).$ Equvalently, $a_{i}c_{i}rt=ta_{i}c_{i}r$, for any $t \in R$ . Therefore $a_{i}rt=ta_{i}r$. So $a_{i}r \in C(R)$. Therefore  $a_{i}s_{i}^{-1} r \in C(RS^{-1})$. Thus $RS^{-1}$ is right Central McCoy. \\
Conversely, assume that $RS^{-1}$ is right Central McCoy. Let $f(x)=\sum_{i=0}^n a_{i}x^{i}$ and  $g(x)=\sum_{j=0}^m b_{j}x^{j} \in R[x]\setminus\{0\}$  such that $f(x)g(x)=0$. Then there exists $rs^{-1} \in RS^{-1}$ such that  $a_{i}rs^{-1} \in C(RS^{-1}),$ since $RS^{-1}$ is right Central McCoy. Equvalently,~$a_{i}rs^{-1}hp^{-1}=hp^{-1}a_{i}rs^{-1}$, for any $hp^{-1} \in S^{-1}R$ . Therefore, $a_{i}rh=ha_{i}r$. So $a_{i}r \in C(R)$. Thus $R$ is right Central McCoy.
\end{proof}

\begin{cor}\label{Corollary}
For a ring $R$, the following are equivalent:\\
(1) $R$ is right (resp. left) Central McCoy ring\\
(2) $R[x]$ is right (resp. left) Central McCoy ring\\
(3) $R[x,x^{-1}]$ is right (resp. left) Central McCoy ring\\
(4) $\frac{R[x]}{(x^{n})}$ is right (resp. left) Central McCoy ring\\
(5) $R[\{x_{\alpha}\}]$  is right (resp. left) Central McCoy ring, where $\{x_{\alpha}\}$ is any set of commuting indeterminates over R.
\end{cor}

\begin{proof}
$(1) \Leftrightarrow (2)$ is due to Theorem 2.6 and $(1) \Leftrightarrow (4)$ is by Theorem 2.9.
 $(1)\Rightarrow (5)$. Let $ F(y), G(y) \in R[\{x_{\alpha}\}][y]$ with $G(y)F(y)=0$. Then $F(y), G(y) \in R[x_{\alpha_{1}},x_{\alpha_{2}},...,x_{\alpha_{n}}][y]$ for some finite subset $\{x_{\alpha_{1}},x_{\alpha_{2}},...,x_{\alpha_{n}}\} \subseteq \{x_\alpha\}$. Following $"(1)\Rightarrow (2)"$ and by induction, the ring $R[x_{\alpha_{1}},x_{\alpha_{2}},...,x_{\alpha_{n}}]$ is right Central McCoy, so there exists nonzero $h_{1} \in R[x_{\alpha_{1}},x_{\alpha_{2}},...,x_{\alpha_{n}}]  \subseteq R[\{x_\alpha\}]$ such that $F(y)h_{1}=0$. Hence, $R[\{x_\alpha\}]$ is right Central McCoy.
 
 $(5)\Rightarrow (1)$ is similar to $"(2) \Rightarrow (1)"$.
 
 $(2) \Leftrightarrow (3)$ Let $S={1,x,x^{2},...}$. Then clearly $S$ is multiplicatively closed subset of $R[x]$ consisting entirely of central regular elements. Since $R[x,x^{-1}]=R[x]S^{-1}, R[x,x^{-1}]$ is right central McCoy if and only if $R[x]$ is right Central McCoy by Theorem 2.12.
\end{proof}

  \singlespacing

\end{document}